\documentclass[a4paper,11pt]{article}

\usepackage[plainpages=false]{hyperref}
\usepackage{amsfonts,latexsym,rawfonts,amsmath,amssymb,amsthm, mathrsfs, lscape}
\usepackage{verbatim}

\usepackage[all]{xy}
\usepackage{authblk}
\usepackage{color}

\usepackage{array, tabularx}

\usepackage{setspace}
\usepackage[top=1.5in, bottom=1.5in, left=1.3in, right=1.3in]{geometry}

\newtheorem{thm}{Theorem}[section]
\newtheorem{cor}[thm]{Corollary}
\newtheorem{lem}[thm]{Lemma}

\newtheorem{prop}[thm]{Proposition}
\newtheorem{prob}[thm]{Question}
\newtheorem{conj}[thm]{Conjecture}
\newtheorem{rmk}[thm]{Remark}

\numberwithin{equation}{section}

\def \O {\mathcal O}

\def \Q {\mathbb Q}
\def \R {\mathbb R}
\def \C {\mathbb C}
\def \Z {\mathbb Z}
\def \P {\mathbb P}

\def \A {\mathcal {A}}

\begin{document}

\title{Explicit Gromov-Hausdorff compactifications of moduli spaces of K\"ahler-Einstein Fano manifolds}

\author{Cristiano Spotti\thanks{c.spotti@qgm.au.dk} $^1$ and  Song Sun\thanks{song.sun@stonybrook.edu} $^2$  \vspace{12pt}\\\vspace{5pt}
	{\em {\footnotesize{Dedicated to Sir Simon Donaldson on his 60th​ birthday  \vspace{12pt}}}}}

\affil{$^1$QGM Aarhus University\\$^2$Stony Brook University}

\maketitle

\begin{abstract} 
	We exhibit  the first non-trivial concrete examples of Gromov-Hausdorff compactifications of moduli spaces of  K\"ahler-Einstein Fano manifolds in all complex dimensions bigger than two (Fano K-moduli spaces). We also discuss potential applications to explicit study of moduli spaces of K-stable Fano manifolds with large anti-canonical volume.  Our arguments are based on recent progress about the geometry of metric tangent cones and on related ideas about the algebro-geometric study of singularities of K-stable Fano varieties.
	 \end{abstract}

\vspace{5 mm}


\section{Introduction}

The understanding  of \emph{moduli spaces of Einstein metrics} on smooth manifolds, together with the process of formation of singularities, is a difficult task still  far from being fully  achieved. However, in certain cases the picture become clearer. Under non-collapsing assumptions, Cheeger-Colding theory of limit spaces \cite{CC}, which generalizes results of Anderson \cite{A}, Tian \cite{T90} and others previously obtained in real four dimension, provides important insights on the limit singularities. Moreover, when the Einstein metrics are also K\"ahler, one can use information from complex geometry to study such moduli spaces. 

Moduli spaces of  bioholomorphic isometry classes of complex $n$-dimensional  K\"ahler-Einstein (KE) manifolds with \emph{positive} scalar curvature, which we denote by $\mathcal{E}^+M$, are always pre-compact in the Gromov-Hausdorff (GH) topology. By Cheeger-Colding theory, degenerate limits are non-collapsed spaces, singular in Hausdorff codimension at least four. Moreover, such limits are naturally \emph{projective KE Fano varieties}, that is,  normal algebraic varieties with positive  anticanonical divisor carrying singular KE metrics where the metric and algebraic singularities match precisely \cite{DS14}.  Thus we can define a fairly abstract \emph{GH compactification}  $\overline{\mathcal{E}^+M}^{GH}$ parameterizing certain smoothable Fano varieties. More precisely, thanks to the equivalence between the existence of KE metrics and the algebro-geometric notion of K-stability \cite{CDS}, which also holds in the singular smoothable setting \cite{SSY}, one can define by forgetting the metric structure  a natural ``Hitchin-Kobayashi map'' $$\phi: \overline{\mathcal{E}^+M}^{GH}\longrightarrow  \overline{\mathcal{K}M},$$ where the set $\overline{\mathcal{K}M}$ denotes the isomorphism classes of ($\Q$-Gorenstein smoothable) K-polystable Fano varieties. It can be shown that $\overline{\mathcal{K}M}$ admits a natural  structure of complex analytic space for which the above map $\phi$ become a homeomorphism \cite{LWX1, O15, OSS}, and  \cite{S} for a survey. 

These  abstract  \emph{Fano K-moduli spaces} $\overline{\mathcal{K}M}$  generalize analogous compactifications of moduli of varieties with positive canonical bundle (KSBA compactifications \cite{Kollar}), and thus they are important spaces to study also for purely algebro-geometric reasons.

However, very few  \emph{explicit examples}  of such compact  Fano KE/K-moduli spaces have been known: the only non-trivial ones are the complex two dimensional orbifold compactifications pioneered in \cite{MM} and completed in \cite{OSS}. In this article, we begin an investigation of some higher dimensional cases.  We are mainly interested in two situations. One is the hypersurfaces/complete intersections cases: here, on one hand, there are known results on classical GIT constructions of moduli compactifications and, on the other hand, we often know that the KE moduli is non-empty.  The other situation is the three dimensional case, in which there is a complete classification of \emph{smooth} Fanos by works of Iskovskikh, Mori and Mukai (see the reference book \cite{IP}).

More precisely, first natural examples to consider in higher dimensions are given by \emph{del Pezzo manifolds}, i.e., $n$-dimensional Fano manifolds such that $K_X^{-1}=L^{n-1}$  with $L$ an ample line bundle, in other words, with \emph{Fano index $n-1$}. Smooth del Pezzos and, as it will be important later on, mildly singular ones,  are classified  by T.  Fujita \cite{TFuj}. When the del Pezzo degree $d:=c_1^n(L)$ is less than or equal to four there are non-trivial moduli deformations. For $d=4$ all such smooth del Pezzo manifolds are  realized as smooth complete intersection of two quadrics $X=Q_1 \cap Q_2$ in  $\C \P^{n+2}$, giving an $n$-dimensional  generalization of degree four Del Pezzo surfaces studied in \cite{MM, OSS}. Such Fanos form, up to biholomorphism, a family of deformations which we will denote as $M_{dP_4^n}$ and they \emph{all} admit KE metrics by \cite{AGP}. An algebro-geometric compactification of $M_{dP_4^n}$ is given by a standard GIT quotient $\overline{M}_{dP_4^n}^{GIT}$ naturally associated to their defining embedding \cite{AL} (compare Section \ref{dP4}).  Its boundary GIT polystable points are well-understood.

  We are now ready to state our first main result.

\begin{thm} \label{MT1} For any dimension $n\geq 2$, the GH compactification $\overline{\mathcal{E}^+M}^{GH}_{dP_4^n}$ of the moduli space of KE del Pezzo manifolds of degree four is naturally  identified with the classical GIT quotient  $\overline{M}_{dP_4^n}^{GIT}$. Hence, $\overline{M}_{dP_4^n}^{GIT}\cong\overline{\mathcal{K}M}_{dP_4^n}$.
\end{thm}

We will discuss some properties of these compact moduli spaces and of the varieties they parameterize in Section \ref{dP4}. In particular, this result  implies an explicit classification of which possibly singular complete intersection of two quadrics in $\C\P^{n+2}$ admit KE metrics (see Theorem \ref{KEQ}),  thus extending and reproving the results for the smooth case in \cite{AGP}. Moreover, thanks to the matching between algebraic and metric singularities, it gives a concrete description of all singularities in the GH limits (see Corollary \ref{sing}).  However, the precise asymptotic metric behaviors near the singular loci remain unsettled, but conjecturally clear.

Similarly to \cite{OSS}, the main idea behind the proof of Theorem \ref{MT1} is based on finding a good preliminary \emph{a-priori} understanding of GH limits via the study of their singularities, and then use some classification results for Fano varieties combined with stability comparison arguments to show that a GH limit of KE intersections of quadrics must be a standard GIT-polystable intersection too. Thus the main content of this paper is to find the right constraints on the singularities (in general not of quotient type) of the metric limits.  To obtain this, we make use of the recent algebro-geometric understanding of metric tangent cones via the two steps construction \cite{DS15} $$ \mbox{sing. germ $(Z,p)$ $\rightsquigarrow$ weighted algebraic tangent cone $W$ $\rightsquigarrow$ metric tangent cone $C(Y)$},$$ combined with the  improved volume bound for K-stable Fano varieties recently obtained by Liu \cite{Liu} (after works of Berman \cite{B},  K. Fujita \cite{KFuj} and Li \cite{Li16}) in term of normalized volume of valuations. 

Another motivation for the writing to this article is related to finding \emph{explicit examples} of K\"ahler-Einstein Fano manifolds. By the result in \cite{CDS} this amounts to verify the K-stability for a given Fano manifold. There have been  important recent progresses towards this, see for example \cite{FO, Fujita1, Fujita2, PW}  as well as older results  based on $\alpha$-invariant/log canonical threshold and multiplier ideal sheaves techniques, for example \cite{Siu, T90, TY, N, DK, Ch, CW},  or with large symmetries, e.g.,  \cite{WZ, Del, IS}. However, in general, a systematic understanding is not yet accomplished. 

The study in this paper leads to a different strategy via moduli spaces, which originates from \cite{OSS} and, when it works, it produces not only one example of K-stable Fano manifold, but a \emph{whole explicit  family of them}. This strategy seems particularly appealing for Fano $3$-folds, thanks to the above mentioned complete classification of smooth ones \cite{IP}.  In essence the basic strategy consists in first showing that a given Fano in a family is KE and then run a ``moduli continuity method"  to show that the set of KE Fano varieties is both open (via small deformations argument) and closed (via a-priori bound on GH limits and stability comparison) in a given explicit GIT quotient or, more generally, glueings of GIT quotients.

 We make some first steps along this direction, mostly under the natural hypothesis of an  explicit \emph{gap conjecture} on the volume density of Ricci-flat K\"ahler cones, which is the ratio between the volume of a ball of radius one centered at the vertex and the corresponding ball in the Euclidean space of the same dimension. We believe that such conjecture holds, especially for small values of $n$. Notice that, for our interest, in this paper we always assume the dimension to be $n\geq 3$.

 \begin{conj}[ODP volume gap conjecture in dimension $k$]\label{DGODP}
Let $C(Y)$ be a Ricci-flat K\"ahler cone in dimension $k$ with an isolated singularity at the vertex. Then the volume density of $C(Y)$ is at most the volume of the $k$-dimensional Stenzel's cone, which is the ordinary double point (ODP) $\sum_{i=1}^{k+1} x_i^2=0$ in $\C^{k+1}$ endowed with the explicit Ricci-flat K\"ahler cone metric.

 \end{conj}
 
The standard Bishop-Gromov theorem implies that the volume density is at most $1$, and it is also well-known that there is a definite but abstract gap from $1$ (compare Lemma \ref{lem5-1}). One can also reinterpret the statement as that the link of the ODP singularity has the maximal volume of among all \emph{smooth} Sasaki-Einstein manifolds in dimension $2k-1$ that are not isometric to the sphere.    We will say more about this in Section \ref{OtherKE} and also formulate  a corresponding algebro-geometric statement.

\begin{thm} \label{thm1-3}
Assuming the Conjecture \ref{DGODP} is true for all $k\leq n$, then it holds: 

\begin{enumerate}
\item For smooth  KE Fano $n$-folds with anti-canonical volume  $c_1^n(-K_X)> \left(\frac{n^2-1}{n}\right)^n$,  the GH moduli closure must contain only Fano varieties with at worst Gorenstein canonical singularities.

\item The GIT moduli of cubic $n$-folds coincides with the KE/K-moduli. In particular, all  smooth cubics admit KE metrics. 

\end{enumerate}

\end{thm}

 We remark that the above  results will be direct consequences of a more general statement independent of the conjecture (compare Theorem \ref{thm5-2}).
 
In dimension three there is a partial extension of the classification to certain  Gorenstein Fano $3$-folds by Mukai \cite{Mukai}. It seems possible that one can use similar results in combination with our discussion  to understand further examples of Fano K-moduli spaces. Most examples of Fano $3$-folds are within the above volume bound, e.g., the moduli space of Iskovskikh-Mukai-Umemura Fano $3$-folds. It is an interesting open problem to understand \emph{on their whole moduli} which of them admit KE metrics. As we will discuss in Section \ref{OtherKE}, the moduli techniques explained in this article suggest a possible strategy for this and other similar situations.

\subsection*{Acknowledgements}  Both authors would like to thank deeply Sir Simon Donaldson for educating them on K-stability and K\"ahler-Einstein metrics over many years. We are also grateful to Kento Fujita, Radu Laza, Yuchen Liu, Yuji Odaka and Jason Starr for helpful discussions.

C.S. is partially supported by AUFF Starting Grant 24285. S.S. is partially supported by NSF grant DMS-1405832 and an Alfred P. Sloan Fellowship.

 \section{Metric tangent cones, volume densities and valuations} \label{prelim}
  In this section we review and collect some previously known results.  Let $Z$ be a GH limit of an $n$-dimensional KE Fano manifolds $(X_i,\omega_i)$ with fixed volume $V=c_1^n(-K_{X_i})$. By  general theory \cite{DS14}, we know that $Z$ is a naturally a  smoothable $\Q$-Fano variety, that is, a normal variety whose canonical bundle is $\Q$-Cartier and with at worst Kawamata log terminal (klt) singularities. Moreover, it carries a weak KE metric $\omega_\infty$ in the sense of pluri-potential theory \cite{EGZ} with $Vol(Z)= c_1^n(-K_{Z})=V$.
  
   Let $p$ be a singular point in the GH limit, recalling that the set of singular points of the underlying algebraic $\Q$-Fano variety coincides with the set where the limit metric is not smooth. A \emph{metric tangent cone} $C(Y)$ at $p$ is by definition a pointed GH limit centered at $p$ of $(Z,  a_i \omega_\infty,p)$ for some increasing sequence of positive numbers $a_i \rightarrow +\infty$. A priori $C(Y)$ may depend on the chosen scaling sequence, but in this case it has been proved in \cite{DS15} that the tangent cone is actually unique, and it can be realized by some affine complex variety admitting a singular CY cone metric, on the smooth part written as $g_{C(Y)}=dr^2+r^2g_Y$, with $Y$ its $2n-1$ real dimensional link, carrying a (possibly singular) Sasaki-Einstein metric. Similarly one can talk about \emph{iterated tangent cones}, that is, metric tangent cones at singular points of the tangent cones, and further iterations. The theory described below applies to these as well. 
   
The cone $C(Y)$ can be obtained from the analytic germ  $(Z,p) \subset Z$  near $p$ via the following two steps procedure, which we now recall.

 Let $r_{KE}(x)=d_{KE}(x,p)$ be the distance function   from $p$ with respect to the limit KE metric. Then one can define the order of vanishing  with respect to the KE metric of a germ of holomorphic function $f \in \mathcal{O}_p$ to be equal to 
 $$\nu_{KE}(f)= \limsup_{r \rightarrow 0^+}  \frac{\log \left( \sup_{r_{KE}\leq r}|f| \right)}{ \log(r)}.$$
 It is proved in \cite{DS15}, that the above is actually a limit, and $\nu_{KE}(f)$ takes values in a discrete set $S(C(Y)) \subseteq \R_{\geq 0}$ called \emph{holomorphic spectrum} of the tangent cone $C(Y)$. Ordering the element of $S(C(Y)) \subseteq \R_{\geq 0}$,  one has a filtration of ideals  $\mathcal{O}_p=\mathcal{I}_0\supset \mathcal{I}_1\supset \mathcal{I}_2 \dots $ and an associated ring $R_p=\oplus_k \mathcal{I}_k/\mathcal{I}_{k+1}$. Then $W:=Spec(R_p)$ is a well defined normal affine complex variety. Moreover, the analytic germ $(Z,p)$ admits a flat degeneration to  $W$, which further admits an equivariant degeneration to $C(Y)$. Furthermore, it is conjectured in \cite{DS15} that both $W$ and $C(Y)$ should be uniquely characterized in terms of the algebraic geometry of the germ at $p$. 
 
 Going back to more Riemannian consideration, thanks to Bishop-Gromov monotonicity, which also holds for GH limits and their (iterated) tangent cones by Colding's volume convergence theorem \cite{Colding}, we can define the \emph{volume density} at any point $p\in Z$:
 $$\Theta(Z,p)= \lim_{r \rightarrow 0^+} \frac{Vol(B_p^{KE}(r))}{\omega_{2n} r^{2n}} \in (0,1],$$
 where $\omega_{2n}$ denotes the volume of the unit ball in the flat $\C^n$. Here are some well-known properties:
 \begin{itemize}
 	\item $ \Theta(Z,p)= \Theta(C(Y))= \frac{Vol(Y)}{Vol(S^{2n-1}(1))}$, where $Y$ is the link of the (unique) metric tangent Calabi-Yau cone $C(Y)$ at $p$. 
 	\item $ \Theta(Z,p)=1$ if and only if $p\in Z$ is smooth.
 	\item $ \Theta(Z,p)$ is a lower semi-continuous function of $p$. The same  is true for densities on the (iterated) metric tangent cones.
 \end{itemize}

 Next we need to introduce  further algebro-geometric inputs (see \cite{Li16} for details). Given a \emph{valuation} $\nu$ centered at $p\in Z$, one can consider its \emph{normalized volume} 
 $$\widehat{vol}(\nu)=A^n(\nu)\lim_{k\rightarrow +\infty} \frac{n^n}{k^n}\dim \mathcal{O}_p/\mathcal{J}_k ,$$
where the ideal $\mathcal{J}_k=\{f \in \mathcal{O}_p \, | \, \nu(f)\geq k\}$, and $A(\nu)$ the log-discrepancy of $\nu$.

 Generalizing K. Fujita's bound of the volume of smooth KE manifolds \cite{KFuj} (algebro geometrically proved via the notion of Ding semistability) to the singular setting, Liu has recently shown the following optimal volume bound in term of the normalized volume of valuations.
 \begin{thm}[Liu's volume estimate, \cite{Liu}] Let $X$ be a K-semistable $\Q$-Fano variety. Then
 	$$c_1^n(-K_{X})\leq \left(1+\frac{1}{n}\right)^n \widehat{vol}(\nu), $$
 	for any valuation $\nu$ centered at $p\in X$.
 	
 \end{thm}
 
 In particular, one can consider in the above inequality the valuation induced by the order of vanishing $\nu_{KE}$, which appeared in the above recalled definition of weighted tangent  cone $W$, since a GH limit $Z$ is indeed K-polystable by Berman's result \cite{B}.  It is conjectured in \cite{Li16}, and partially verified \cite{LL}, that such valuation minimizes the normalized volume. 
 
 A crucial  observation for us is that the volume of such KE valuation $\nu_{KE}$ is related with the volume density, as shown in \cite{HS}, Appendix C.
 \begin{prop} \label{val} \cite{HS} Let $p$ a point in KE $\Q$-Fano variety $Z$. Then the following inequality holds: $$n^n \Theta(Z,p) \geq \widehat{vol}(\nu_{KE}).$$
 \end{prop}
 
 We remark that the above inequality is actually an  equality if the metric tangent cone $C(Y)$ at $p$ is quasi-regular, and it is conjectured in \cite{HS} that it is always an equality even if the metric tangent cone is irregular. 
 
 Thus the above two results combine to provide an estimate of the total volume of a $n$-dimensional KE $\Q$-Fano variety $Z$ in terms of the volume of $\C\P^n$, with the Fubini-Study metric with the same Einstein constant:
 
 \begin{equation}\label{eqn2-1}
 \Theta(Z,p) \geq \frac{Vol(Z)}{Vol(\C\P^n)}=\frac{c_1^n(-K_Z)}{c_1^n(-K_{\P^n})},
 \end{equation}
 for any $p \in Z$.
 
 By the semi-continuity of densities, the same inequality also holds with the left hand side equal to the volume densities $\Theta(C(Y'))$ for any iterated tangent cone $C(Y')$ at $q\in C(Y)$, and further iteration. Strict inequality between densities  holds if the iterated tangent cone is not isometric to the original tangent cone itself by the rigidity of Bishop-Gromov, see Lemma \ref{rig}.

 Finally, we remark that such estimate should be compared to the analogous one we could obtain using differential geometric argument based on Bishop-Gromov theorem. However, such estimate (based on comparison with the $2n$-dimensional round sphere) is weaker then the above one obtained via algebro-geometric techniques. For example, one can check that in our three dimensional case, a differential geometric estimate would be given by $c_1^3(-K_{Z}) \leq 100 \, \Theta(Z,p)$, but clearly $c_1^3(-K_{\C \P^3})=64<100$. Having this improved bound is essential for our next investigation.
 
 \section{Local Cartier index bound and volume density}\label{indexbound}

Let $V$ and $r$ be positive integers, and denote by $\mathcal K_{n, V, r}$ be the set of all $n$-dimensional K\"ahler-Einstein Fano manifolds $(X, \omega)$ with $c_1^n(-K_X)=V$ and $K_X^{-1}=L^{r }$ for an ample line bundle $L$. Let $Z$ be a GH limit of elements in $\mathcal K_{n, V, r}$.
 
  Since we have smooth convergence on the smooth locus $Z^{sm}$ of $Z$,  there is a limit Hermitian holomorphic line bundle $L_Z$ on $Z^{sm}\subset Z$ with $L_{Z}^{r}=K_{Z}^{-1}$. Moreover, by \cite{DS14}  $L_Z$ extends as a $\Q$-line bundle on the whole $Z$. 
 Namely, there is an integer $\tilde r$ such that $L_Z^{\tilde r}$ extends as a holomorphic line bundle to  $Z$. In other words, $L_Z^{\tilde r}$ is Cartier on $Z$. 
  At each singular point $p$ of $Z$, we define the local \emph{Cartier index} $ind(p, L_Z)$ to be the smallest integer $\tilde r$ such that $L_Z^{\tilde r}$ is Cartier near $p$, and let $ind(p)=ind(p, L_Z)/r$ be the local \emph{Gorenstein index} of $p$, so that $K_Z^{-ind(p)}$ is Cartier near $p$. 
  
 More generally,  if we choose a sequence of points $p_i\in X_i$ that converge naturally to $p$ under the GH convergence, and choose a sequence of positive integers $a_i$, we can discuss \emph{pointed rescaled GH limits} of $(X_i, a_i\omega_i, p_i)$.  It is shown in \cite{DS15}, Theorem 1.4 that such a rescaled limit $(Z', p')$ if non-compact (i.e., when $a_i\rightarrow\infty$),  is an affine algebraic variety with log terminal singularities which is endowed with a weak Ricci-flat K\"ahler metric. On $Z'$, there is a limit Hermitian $\Q$-line bundle $L_{Z'}$ which is a genuine line bundle on the smooth locus,  and one can similarly define the local  Cartier index and the Gorenstein index at a singular point of $Z'$. Notice that if $a_i\rightarrow\infty$, then the limit Chern connection on $L_{Z'}$ over the smooth locus of $Z'$ is flat (since the curvature of the connection is fixed but the K\"ahler metric is rescaled by $a_i$). In general the flat connection may not be trivial since we do not know a-priori its holonomy, which depends on the local topology around the singular locus. This is a key technical issue that has to be overcome in the proof of the main result in \cite{DS14}, and the idea there was to raise the power of the line bundle to make the holonomy sufficiently small, which suffices for \cite{DS14}. For our main purpose here however, we need to get \emph{explicit} estimates of the power needed.

Let $(Z', p)$ be a pointed rescaled GH limit, which may be $Z$ itself (i.e., when $a_i$ is bounded). One can take a metric tangent cone $C_p:=C(Y)$ at $p$. Notice that, by a diagonal argument, $C_p$ itself is also a pointed rescaled limit.  Let $O_p$ be the vertex of $C_p$, then  we have
 
 \begin{lem} \label{ind}
  $ind(p, L_{Z'})=ind(O_p, L_{C_p})$.
 \end{lem}
 
 \begin{proof}

 	We first prove that $ind(p, L_{Z'})\geq ind(O_p, L_{C_p})$. Let $r_1=ind(p, L_{Z'})$, then there is a local non-vanishing holomorphic section $s$ of $L_{Z'}^{r_1}$ in a neighborhood of $p$. Using the local gradient estimate of holomorphic sections as in \cite{DS15} we know $|\nabla s|$ is uniformly bounded in a neighborhood of $p$. As we rescale the metric, we see the derivative of $s$ tends to zero, so that $s$ gives rise to a parallel section of $L_{C_p}^{r_1}$ on the smooth part of $C_p$. In particular $ind(O_p, L_{C_p})\leq r_1$.\

 	For the other direction it seems likely to follow from general algebro-geometric facts,   but here we point out that it also follows directly from the results of \cite{DS14, DS15}. Let $r_0=ind(O_p, L_{C_p})$. Given a local non-vanishing holomorphic section $s$ of $L_{C_p}^{r_0}$ in a neighborhood of $O_p$, one can apply the H\"ormander construction to obtain a local non-vanishing holomorphic section of $(L_{Z'})^{r_0}$ around $p$. This can be proved in exactly the same way as in  Lemma 2.5 and Proposition 2.9 of \cite{DS15}. One can simply replace the holomorphic function there by the holomorphic section $s$. The point is that locally one can realize $(Z', p)$ as GH limits of smooth K\"ahler-Einstein manifolds, and one can use $s$ to construct local holomorphic sections with uniform lower bound on these smooth manifolds by $L^2$-estimate on Stein manifolds. Then by applying a diagonal argument we obtain the required section of $L_{Z'}^{r_0}$ on a neighborhood of $p$ in $Z'$. 
	\end{proof}
 
 \begin{rmk} \label{rmk3-2}
 We have indeed proved more: on the tangent cone $C_p$, the flat connection on $L_{C_p}^{r_1}$ is trivial and in particular $L_{C_p}^{r_1}$ is holomorphically trivial. Here $r_1=ind(p, L_{Z'})$. 
 \end{rmk}
 
 Let $\mathcal I_p$ be the set of all iterated metric tangent cones at $p$, i.e., we include $C_p$, the set of all the tangent cones of all points in $C_p$,  all the further tangent cones, and so on. As we explained in Section \ref{prelim},  we have that  $\Theta(C)\geq V/(n+1)^n$ for all $C$ in $\mathcal I_p$. Recall from general convergence theory of Riemannian manifolds that each element in $\mathcal I_p$ is a geodesic length space, meaning that given any two points there is a \emph{length minimizing geodesic} whose length agrees with the distance between the two points. By Cheeger-Colding theory, any element $C(Y)$ in $\mathcal I_p$ admits a regular-singular decomposition which is invariant under radial dilation. By the result of Colding-Naber \cite{CoNa} the regular locus of $C(Y)$ is \emph{geodesically convex}, namely, any two regular points can be connected by a length minimizing geodesic that consists only of regular points. Indeed, a stronger statement is proved in \cite{CoNa}: given one  regular point and another point, then one can find a length minimizing geodesic connecting them, all points on which are regular except possibly one end point. For convenience we call this property  \emph{strong geodesic convexity}.

 \begin{lem} \label{lem3-3}
 	For any cone $C\in \mathcal I_p$ which is  isometrically of the form $\C^k \times C(Y)$ where $k\geq 0$, we have $\pi_1(Y)$ is finite and its order is bounded by $\Theta(C)^{-1}$.
 \end{lem}

 \begin{proof}
It is elementary to see that given any two point $q_1, q_2$ in $Y$ with distance $D$ smaller than $\pi$, a length minimizing geodesic connecting $(0, q_1)$ and $(0, q_2)$ in $\C^k\times C(Y)$ is of the form $(0, c(t)\gamma(t))$ for a length minimizing geodesic $\gamma(t)$ in $Y$ connecting $q_1$ and $q_2$ and a function $c(t)$ that depends only on $D$ (which can be determined by considering the model case $Y=S^1$).  This implies that the regular set of $Y$ is strongly geodesically convex. 

Let $\pi: \widetilde Y\rightarrow Y$ be the universal cover. Then $\widetilde Y$ is also naturally a geodesic length space, and a length minimizing geodesic on $\widetilde Y$ locally maps to length minimizing geodesic in $Y$. We claim that the regular part of $\widetilde Y$ is also strongly geodesically convex. To see this, take a regular point $q_1$ and an arbitrary point $q_2$ in $\widetilde Y$, and let $\widetilde \gamma(t)(t\in [0, 1])$ be a length minimizing geodesic connecting them. If $\widetilde\gamma(t)(t\in [0, 1))$ does not contain entirely of regular points, then since the regular set is open, we can find the first time $T\in (0, 1)$ such that $\widetilde \gamma(T)$ becomes singular. Let $\delta>0$ be such that any ball of radius $2\delta$ in $\widetilde Y$ is mapped isometrically to a ball in $Y$.  Then we consider $\widetilde\gamma(t)$ for $t\in [T-\delta, T+\delta]\cap [0, 1]$. Denote by $\tilde T=\min(T+\delta, 1)$. We know $\pi(\widetilde\gamma(T-\delta))$ is regular, so by the strong geodesic convexity of $Y$ it follows that we can find a possibly different length minimizing geodesic $\gamma'(t)$ connecting $\pi(\widetilde \gamma(T-\delta))$ and $\pi(\widetilde \gamma(\tilde T))$ so that $\gamma'$ is regular in the interior. Now by our choice of $\delta$ we can lift $\gamma'$ to a length minimizing geodesic $\widetilde\gamma'$ in $\widetilde Y$ connecting $\widetilde \gamma(T-\delta)$ and $\widetilde \gamma(\tilde T)$.  So we can replace the portion of $\widetilde\gamma(t)$ over $[T-\delta, \tilde T]$ by $\widetilde\gamma'$, and still obtain a length minimizing geodesic but on which the first singular time is at least $\tilde T$. Now we can continue this process and since $\delta>0$ is a fixed number we can in the end reach a length minimizing geodesic on which there is no singular point in the interior. This proves the claim.
  
By Cheeger-Colding-Tian \cite{CCT} it follows  that $k\leq n-2$, and the regular locus of $\widetilde Y$ is Einstein with Ricci curvature $2(n-k)$. Now since the singular set of $\widetilde Y$ has zero volume (since it has Hausdorff codimension at least $2$), we can adapt the usual proof of Bishop-Gromov volume comparison theorem to $\widetilde Y$, and  conclude that the volume of a ball centered at a regular point $q\in \widetilde Y$ is at most the volume of the corresponding ball in $S^{2n-2k-1}$. Therefore the volume of $\widetilde Y$ is at most the volume of $S^{2n-2k-1}$. On the other hand, it is easy to see by definition that the volume of $Y$ is equal to $\Theta(C) Vol(S^{2n-2k-1})$, so we conclude that the degree of the cover is finite,  and is bounded above by $\Theta(C)^{-1}$.
 \end{proof}
 \begin{rmk}\label{rmk3-4}
 Indeed the above arguments prove more. Take a singular point $\tilde q$ of $\widetilde Y$, with image $q\in Y\subset C(Y)$, then it holds that
 \begin{equation} \label{eqn3-1}
 Vol(\widetilde Y)\leq \Theta(C(Y), q) Vol(S^{2n-2k-1}), 
 \end{equation}
where $\Theta(C(Y), q)$ is the density of the tangent cone of $C(Y)$ at $q$.  This follows from the fact that for fixed $r<\delta$, $Vol(B(q, r))$ is a continuous function of $q$ in $\widetilde Y$, and applying Bishop-Gromov volume comparison on balls centered at a regular point of $\widetilde Y$.

 \end{rmk}
 \begin{rmk}
 	We mention that for $n=3$ it is proved in \cite{DS14} that the link of tangent cones are five dimensional Sasaki-Einstein orbifolds, in which case the above arguments can be made more direct since the universal cover $\widetilde Y$ is also a Sasaki-Einstein orbifold. 
 \end{rmk}
 
We can now use the above proposition to study the Cartier index on iterated tangent cones via an induction argument based on the stratification of the singularities.

 \begin{lem} \label{cart}
 	If $\Theta(Z,p)>1/2$, then the $\Q$-line bundle $L_{C}$ is Cartier for all $C\in \mathcal I_p$.
 \end{lem}
 \begin{proof}
 	We prove this by induction on $k$ where $n-k$ is the maximum number of flat factor $\C$ that a tangent cone can splits off holomorphic isometrically. When $k\leq 1$ by \cite{CCT} the tangent cone is $\C^n$ and there is nothing to prove. Suppose the conclusion holds for all $k\leq l$. Now suppose one cone $C$ is of the form $\C^{n-l-1}\times C'$. For each non-vertex point $q$ of $C'$, the metric tangent cone $C_q$ would split off at least a factor $\C^{n-l}$, by the splitting theorem \cite{CCT}. So by induction assumption and Lemma \ref{ind} we know $L_{C'}$ is Cartier near $q$, hence $L_{C}$ is Cartier on the complement of $\C^{n-l-1}\times \{O_{C'}\}$. It then suffices to prove that $L_{C}$ is holomorphically trivial. By Remark \ref{rmk3-2} we know $L_{C}^{r_1}$ is holomorphically trivial with $r_1=ind(p, L_{Z'})$. Let $m$ be the smallest integer so that $L_C^m$ is holomorphically  trivial, and let $s$ be a trivializing section.  Then we can define a degree $m$ covering  of $C\setminus \C^{n-l-1}\times \{O_{C'}\}$ as the subset of $(l, x)$ where $l\in (L_C)_x$ with $l^m=s_x$. This is clearly well-defined on the regular locus and it also extends naturally over the singular locus since $L_C$ is Cartier. The covering is connected by our choice of $m$.  By the  volume bound  we know  $\Theta(C)\geq \Theta(Z,p)>1/2$,   so by Lemma \ref{lem3-3} the set  $C\setminus \C^{n-l-1}\times \{O_{C'}\}$ is simply connected hence $m$ must be equal to $1$. Hence we have finished  the induction proof.
 \end{proof}

 \begin{prop} \label{large}
 If $\Theta(Z,p)>1/2$, then $L_Z$ is a genuine line bundle near $p$, and $p$ is a  Gorenstein canonical singularity. 
 \end{prop}
 
 \begin{proof}
  By the density estimate, we know $\Theta(C)>1/2$ for all possibly iterated tangent cones. Applying Lemmas \ref{ind} and \ref{cart} we see that  $ind(p, L_Z)=1$. Hence $L_Z$ is Cartier. Then so is $-K_X$, hence the Gorenstein index of $Z$ is one. Since $Z$ has log terminal singularities, the discrepancies of any log resolution, which must be integers in our case, are at least $-1$, so they are non-negative. This shows $Z$ has canonical singularities. It is well-known that canonical singularities are Cohen-Macaulay so must be Gorenstein.
 \end{proof}

Using (\ref{eqn2-1}) we immediately obtain the following result on GH limits of KE Fanos with large volumes.
 
\begin{cor} \label{cor3-6}
Suppose $V>\frac{1}{2}(n+1)^n$, then a Gromov-Hausdorff limit $Z$ of KE Fanos in $\mathcal K_{n, V, r}$ has Gorenstein canonical singularities and $K_Z=L_Z^{-r}$ for a genuine line bundle $L_Z$. 
\end{cor}

As we will see below, such condition is not empty since del Pezzo manifold of degree four are within that bound. However, it is not known to us if there exist other Fano $n$-folds ($n\geq 4$) which have \emph{non-trivial K-moduli} and which satisfy the above volume inequality. Thus, we ask the following question:

\begin{prob} Beside the degree four del Pezzo case, are there any other examples of not rigid K-stable Fano manifolds whose volume is bigger than half the volume of the projective space?
\end{prob}

In any case,  one can use exactly the same argument to obtain the following more general Cartier index bound. 
   
   \begin{prop}
   The local Cartier index satisfies $ind(p, L_Z)\leq \Lambda^{n-1}$, where $\Lambda=\left \lfloor1/\Theta_p\right \rfloor$.
   \end{prop} 
   
   \begin{proof}
     By general results from Cheeger-Colding-Tian theory we know that if a metric cone in $\mathcal I_p$ splits off isometrically a factor $\C^{n-1}$ then it must be isometric to $\C^n$. Using this, we can easily follow the proof of Lemma \ref{cart}, and apply (\ref{eqn2-1}) to get the conclusion.
   \end{proof}
   
  \begin{rmk}
Notice this bound is certainly not optimal. One expects the exponent $n-1$ should not appear. We need it here due to purely technical reason since we used the induction argument.  In general at a singular point $p$, similar to the proof of Lemma \ref{cart}, one can always take the covering of the regular part of $C_p$ with degree equal to $ind(O_p, L_{C_p})$, and then this yields a bound of the Cartier index if one understands the metric completion of the covering well-enough to show  the Bishop-Gromov inequality holds.  This is a non-trivial question and we leave it for future study.
  \end{rmk}

 \section{KE moduli of degree four del Pezzos}\label{dP4}
 
 Let $X$ be a smooth $n$-dimensional del Pezzo manifolds of del Pezzo degree $d=4$. Then, by classification, $X$ is realized as a smooth complete intersection of two quadrics $X=Q_1 \cap Q_2$ in $\C\P^{n+2}$.

 Now we recall that, more generally, a \emph{Del Pezzo variety} is a $\Q$-Fano variety with Gorenstein canonical singularities such that $K_X^{-1}=H^{n-1}$ for some ample line bundle $H$. Then clearly $c_1^n(-K_X)=d(n-1)^n$. The classification theory of T. Fujita extends also to this singular case \cite{TFuj}. In particular when $d=4$ we still get intersection of two quadrics in $\C\P^{n+2}$.
 
 By (\ref{eqn2-1}), we see that if $n>3$, $d=4$, then $d(n-1)^n> \frac{1}{2} (n+1)^n$. So by Corollary \ref{cor3-6} we know any GH limit of  K\"ahler-Einstein del Pezzo manifolds of degree 4 in dimension $n>3$ must be again a del Pezzo variety of degree 4, in the above sense. In combination with T. Fujita's classification \cite{TFuj} of Gorenstein canonical Fano variety with large Fano index, we immediately obtain:

 \begin{cor} \label{GH4}
 	For any dimension $n>3$, a GH limit of   KE del Pezzo manifolds of degree $4$ is always an intersection of two quadrics in $\C \P^{n+2}$.  
 \end{cor}

 The case $n=3$ does not follow immediately, since we hit the equality case in the volume estimate \ref{eqn2-1}. To deal with it, some ad hoc consideration, based on rigidity results,  is required.
 
 \subsection{GH limits in the three dimensional case}
 
 Our next goal is to show that Corollary \ref{GH4} also holds for $n=3$. Namely:
 
 \begin{prop} \label{dP43D}
 Let $Z$ be a GH limit of three dimensional smooth K\"ahler-Einstein del Pezzo manifolds of degree $4$, then $Z$ is a del Pezzo $3$-fold of degree $4$.
 \end{prop}
 
Again by Fujita's classification it suffices to show that on any GH limit $Z$, the corresponding limit bundle $L_Z$ is a genuine line bundle on the whole $Z$. 

\

If $p \in Z$ is a singular point in the GH limit with density $\Theta(Z,p)> 1/2$, then Proposition \ref{large} still apply, giving that $L_Z$ is a line bundle near $p$ and that the singularity is Gorenstein canonical.

 	We remark that in this three dimensional case there is another argument. Note that in such case the link of the tangent cone  $Y$ (clearly necessarily smooth, by semicontinuity of densities) is simply connected (again by volume comparison). Moreover by  Cheeger-Tian \cite{CT} and Colding-Minicozzi \cite{CM})  we have $Y\times \R$ is diffeomorphic  a punctured neighborhood of the singularity, so to  $M\times \R$, where  $M$ is to the classical Milnor's link obtained by intersecting the germ  with a small euclidean sphere in an embedding of the singularity.  Hence $M$ is also simply connected, and this is also enough to have our local index bound at $p$.  
 
 \begin{rmk}
 	It may be interesting to note that, even if a-priori it is not clear that the link of the cone $Y$ and the Milnor's link $M$ at the singularity are homeomorphic, the fact that they are actually diffeomorphic  follows easily (see uniqueness statement in the main Theorem of \cite{Br}) by using  the \emph{$h$-cobordism Theorem}. We thanks Marcel B\"okstedt to point such reference to us.
 \end{rmk}
 
 We now assume that $\Theta(Z,p)= 1/2$. Here there are two possible sub-cases:
 
 \subsubsection*{Case 1:  $\Theta(Z,p) = \frac{1}{2}$ and $C(Y)$ has an isolated singularity;}
 
 The same argument as before works to imply that the singularity is Gorenstein canonical and $L$ is a line bundle near $p$, due to the following lemma. 
 
 \begin{lem}
 In this case the link $Y$ is simply connected. 
 \end{lem}
 \begin{proof}
 Applying the Bishop-Gromov volume comparison to the universal cover of $Y$ we obtain that either $Y$ is simply connected, or $\pi_1(Y)=\Z_2$. In the second case the universal cover of $Y$ is isometric to the round sphere $S^5$ with radius one. So the tangent cone $C(Y)$ is the orbifold $\C^3/\Z_2$ where the action is free on $S^5$. 
But this is impossible by Schlessinger's rigidity of quotient singularities \cite{Sc}, since the cone is a degeneration of $W$, which is itself a degeneration of the smoothable local germ $Z$ at $p$.
 \end{proof}

 \subsubsection*{Case 2: $\Theta(Z,p) = \frac{1}{2}$ and $C(Y)$ has non-isolated singularity;}
 
 The first observation is based on the following general rigidity splitting Lemma for CY tangent cones we mention in Section \ref{prelim}.
 
 \begin{lem}\label{rig}
 	Let $ q (\neq 0) \in C(Y)$ be a point where $\Theta(C(Y),q)= \Theta(C(Y),0)$. Then $C(Y)$ splits isometrically as $ C(Y')\times \C$.
 \end{lem} 
 
 \begin{proof}
 	Consider the volume ratio $f(r)=\frac{Vol^{CY}_q(B(r))}{\omega_{2n} r^{2n}}$ centered at $q$. Then $f$ is non-increasing and  $\Theta(C(Y),q)=\lim_{r\rightarrow 0^+} f(r)$. On the other end, the volume density at infinity is always equal to $\lim_{ r\rightarrow +\infty} f(r)=\frac{Vol(Y)}{Vol(S^{2n-1}(1))}$. Thus $f(r)$ is identically constant, i.e., there is a  \emph{volume cone} centered at $q$. But, by Cheeger-Colding theory,  any volume cone is actually a metric cone. Thus there is a \emph{line} going through $q$ and hence, by the almost splitting Theorem, we obtain the result.
 \end{proof}
 
 Applying the Lemma to our case, we have that $C(Y)$ splits as $\C \times \left(\C^2/\Z_2\right)$, i.e., the cone  is locally analytically given by the equation $x^2+y^2+z^2=0$ in $\C^4$.
 
 We now prove the following more general proposition, which will be useful also in the discussion of Section \ref{OtherKE}.
 
 \begin{prop}\label{nonisolated}
 Let $Z'$ be a pointed rescaled Gromov-Hausdorff limit of a sequence of elements in $\mathcal K_{n, V, r}$ (as in the beginning of Section \ref{indexbound}), suppose a point $p\in Z'$ has metric tangent cone $\C^2/\Z_2\times \C^{n-2}$, then the limit $L_{Z'}$ is Cartier near $p$ and hence $p$ is a Gorenstein canonical singularity of $Z'$.
 \end{prop}
 
 We begin with the following Lemma, which is a direct consequence of the arguments in \cite{DS15}. Recall the embedding dimension (denoted by $Embdim(Z, p)$) of a germ of a singularity $(Z, p)$ is the minimal integer $N>0$ such that there is a complex-analytic embedding $F: (Z, p)\rightarrow (\C^N, 0)$.  By general theory (see Grauert-Remmert \cite{GR}), suppose $(Z, p)$ is embedded in some $\C^M$ with $M>Embdim(Z, p)$, then there is a holomorphic function $f$ in a neigborhood of the origin in $\C^M$ that vanishes on $Z$ and has non-vanishing derivative, i.e., $f\in \mathcal I_Z\setminus \mathfrak m_{0}^2$. 
 
 \begin{lem} \label{lem4-7}
 	$$Embdim(Z, p)\leq Embdim(W,0)\leq Embdim(C(Y), 0)$$
 \end{lem}
 \begin{proof}
 	We first prove the the first inequality. Notice $W$ is given by the weighted tangent of $(Z, p)$ under some analytic embedding $F: (Z, p)\rightarrow \C^k$, with respect to some weight $w \in ({\R^+})^{k}$. If $k=Embdim(W, 0)$ then we are done. Otherwise, by the above discussion, we find a function $f\in \mathcal I_{W}\setminus \mathfrak m_0^2$. Using the grading on $W$ we may further assume that $f$ is homogenous with respect to the weight $w$, and $f$ has non-vanishing linear term. From the definition of weighted tangent cone, we may find a holomorphic function $g\in \mathcal I_{Z}$ such that $f$ is the initial term of $g$ with respect to $w$. Clearly $g$ also has non-vanishing linear term.  Say this is $z_k$. Then we may project $(Z, p)$ and $W$ to the $z_1, \cdots, z_{k-1}$ plane, and these projections are indeed holomorphic equivalence onto their images. Hence we have realized $W$ as weighted tangent cone of $(Z, p)$ in a smaller $\C^{k-1}$. We can repeat this until we reach the embedding dimension of $W$, and we are done. 
 	
 	For the second inequality, we notice that there is an equivariant test configuration degenerating $W$ to $C(Y)$. One can work purely in a polynomial level, and essential argue as in the above: if there is a function vanishing on $C(Y)$ with non-vanishing linear term, then there must be another function vanishing on $W$ also with non-vanishing linear term. \end{proof}

 A corollary of  Lemma \ref{lem4-7} is that if $C(Y)$ is a hypersurface, then $Z$ is also a hypersurface singularity. Note that it is in general \emph{not} the case that the GH limit $Z$ is also locally isomorphic to $\C\times \C^2/\Z_2$ near $p$: as suggested by \cite{DS15} and \cite{HN}, a K\"ahler-Einstein metric on the three dimensional (isolated) $A_k$ singularity $x^2+y^2+z^2+w^{k+1}=0$ in $\C^4$ should have metric tangent cone $\C\times \C^2/\Z_2$ when $k\geq 3$.

 By \cite{DS15} and the above lemma we can realize $C(Y)$ as the hypersurface $x^2+y^2+z^2=0$ in $\C^4$ and embed a neighborhood of $p$ in $Z$ analytically into $\C^4$ (with $p$ mapped to $0$), such that $W$ is given by the weighted tangent cone of $Z$ at $p$ with respect to the weight $(2, 2, 2, 1)$ (which corresponds to  the Reeb vector field of $C(Y)$), and $C(Y)$ is a further equivariant degeneration limit of $W$. Now, by taking a generic hyperplane $H$ in $\C^4$ passing $0$, we obtain a deformation of $S_\infty=H\cap C(Y)$, which is locally isomorphic to the $2$ dimensional ordinary double point $x^2+y^2+z^2=0$ in $\C^3$, by the corresponding slices in $Z$. Since the versal deformation of 2d ordinary double point is given by $x^2+y^2+z^2=\epsilon$ which is one dimensional, it follows that a generic hyperplane slice $S_Z$ at $p$ in $Z$ must also have an ordinary double point singularity at $p$, see for example \cite{KS}. 
 
 Now we consider the convergence of $X_i$ to $Z$ inside a projective space \cite{DS14}. We can fit $X_i$ in a flat family (say, over a certain Hilbert scheme) $\mathcal X$ with central fiber $Z$. Similar argument as in the proof of the lemma shows that by restricting to the germ around $p$, we may obtain a corresponding family of germs of hypersurfaces in $\C^4$, which we still denote by $\mathcal X$, that gives rise to a deformation of the germ $Z$ at $p$. By taking generic hyperplane sections again we may assume that the corresponding slices in $\mathcal X$ gives a deformation of $S_Z$. Again by versality of deformations we conclude that for $i$ large by shrinking the size, the slice $S_i$ in $X_i$ must be isomorphic to a neighborhood in the smoothing $x^2+y^2+z^2=1$, given by intersecting with a Euclidean ball in $\C^3$. 
 
 The important fact here is that such a neighborhood is topologically isomorphic to a neighborhood of the zero section in the cotangent bundle of an $S^2$ (the vanishing sphere), hence it must be simply connected. 
 
Now we can prove that $L_Z$ must be a genuine line bundle. As in Lemma \ref{ind} it suffices to show that $L_{C(Y)}$ is the trivial line bundle. To see this we first recall by the result of \cite{DS14} that there is a fixed integer $N>0$ such that $L_{Z}^N$ is a line bundle and it is naturally the limit of $L_i^N$ over $X_i$ under the GH convergence. We can take holomorphic sections $s_i$ of $L_i^N$ in a neighborhood in $X_i$,  that converges to a non-vanishing section $s$ of $L_Z^N$. As in the proof of Lemma \ref{ind} this further gives rise to a parallel section $s_{C(Y)}$ of $L_{C(Y)}^N$.  Since $S_i$ is simply connected, and we can take a $N$-th root of $s_i$, to  obtain a holomorphic section $\sigma_i$ of $L_i$ over $S_i$. Under the natural convergence of $S_i$ to the slice $S_\infty$ (first take limit to $S_Z$, then dilate, and we can take a diagonal sequence), we know $\sigma_i$ converges naturally a holomorphic section $\sigma_\infty$ of $L_{C(Y)}$ restricted to the smooth locus of $S_\infty$ (which is homotopic to the link $S^3/\Z_2$), with $\sigma_\infty^{\otimes N}=s_\infty$. Indeed, because of the scaling procedure the section $\sigma_\infty$ (also $s_\infty$) is parallel section of the flat bundle $L_{C(Y)}$.  This implies the flat connection on $L_{C(Y)}$ restricted to one link $S^3/\Z_2$ has trivial holonomy, hence it is the trivial flat connection outside the singular locus of $C(Y)$. So the corresponding holomorphic line bundle $L_{C(Y)}$ is also trivial. This concludes the proof of Proposition \ref{nonisolated}.

 \

 Finally, we remark that there is another argument to conclude such $3$-dimensional case.  If $p$ is an isolated singularity of $Z$, then by  Milnor's Theorem on the homotopy groups of hypersurfaces singularites \cite{M} Theorem 5.2 we know the topological link of $p$ has trivial fundamental group. Then one can argue similarly as we did several times before to see that $L$ must be holomorphically trivial on the link, hence must be a line bundle near $p$.  If $p$ is an non-isolated singularity of $Z$, then again by the above slice argument we know a generic nearby singularity must locally be given by $\C\times \C^2/\Z_2$. Now we can apply the rigidity case of the Liu's estimate in \cite{Liu} Theorem 3, to conclude that $Z=\C\P^3/\Z_2$, with  $[x:y:z:t] \simeq [x:y:-z:-t]$.  Moreover $\C\P^3/\Z_2$ embeds in $\C\P^5$ via the map $[x:y:z:t] \mapsto [x^2:y^2:xy:z^2:t^2:zt]$ as the unique, up to linear transformation, GIT-polystable  intersection of two quadrics with non-isolated singularities.  
 
 \
 
 In conclusion, by combining the results of this section, we have proved Proposition \ref{dP43D}.

 \subsection{Moduli identification}

 In the previous sections we have shown that all GH limits $Z$ of smooth KE intersection of quadrics are themselves embedded in $\C\P^{n+2}$ as  intersection of two quadrics (and the limit root $L$ of the anticanonical bundle is very ample). 
 
 By considering the pencil spanned by the two quadrics (i.e., taking the second Hilbert point) one can consider the following GIT picture for the natural Pl\"ucker linearization:
 $$SL(n+3, \C)  \curvearrowright Gr(2, Sym^2(\C^{n+3})) ֒\hookrightarrow \P (\wedge^2 Sym^2(\C^{n+3})).$$
 Let $\overline{M_{dP_4^n}}^{GIT}$ such GIT quotient. It is well-known \cite{AL}, that strictly GIT stable points correspond precisely to smooth quadric intersections, thus forming  a  moduli space $M_{dP_4^n}$, itself of dimension $n$, while more general polystable boundary points corresponds  to certain  simultaneously diagonalizable intersections of two quadrics. The study of intersection of quadrics, and questions of simultaneous diagonalization (see  for example Reid thesis \cite{Reid}), is an old subject. Here we mention that the study of the above GIT quotient is in fact related  to the study of binary forms, which can be naturally associated to a pencil by taking its discriminant (see \cite{MM} and \cite{AL}).

  Our next goal is to show that all GH limits of smooth KE intersections of quadrics correspond to GIT polystable points. Using K-polystability of the limit $Z$, which holds by \cite{B}, we can say more, via a stability comparison argument based on the CM line bundle \cite{PT}. Since the proof is the same as Theorems 3.4 and 4.2 in \cite{OSS}, we just sketch the main points. 
  
  Moreover, note that we know  $\mathcal{E}^+M_{dP^n_4}$ is not empty \cite{N}.
 
 \begin{lem} $Z$ is GIT-polystable.
 \end{lem}
 \begin{proof}
 	Being the Picard group of the Grassmanian parameter space isomorphic to $\Z$ and the action an  $SL(n+3,\C)$ action, the CM line bundle (whose weight is the Donaldson-Futaki invariant) on such Grassmanian is equivariantly equivalent up to positive scaling to the Pl\"ucker linearization, since we know that there is at least one K-polystable point. Hence K-polystability implies GIT-polystability.
 \end{proof}

 For readers' convenience, let us sketch the moduli continuity type argument used to prove our main theorem. By the above results we can now define a natural map $$\phi_{dP^n_4}: \overline{\mathcal{E}^+M}^{GH}_{dP^n_4} \rightarrow \overline{M}^{GIT}_{dP^n_4}.$$ Such map is continuous with respect to the GH, and analytic topology of the target. This follows by Luna's slice Theorem (or \cite{LWX1, O15}): if $(X_i,\omega_i)$ GH converge to $X_\infty$, then $[X_i]$ are represented by GIT polystable points near $[X_\infty]$, since the varieties $X_i$ converge to $X_\infty$. Hence $[X_i]\rightarrow [X_\infty]$. The continuity of  $\phi_{dP^n_4}$  extends to the boundary. Now  we claim that $\phi_{dP^n_4}$ is actually a homeomorphism. Injectivity follows by Bando-Mabuchi uniqueness and its generalization to the singular setting \cite{Bern}. Since $Aut(X)$ is finite for a smooth $X$, by Implicit Function Theorem  $\phi_{dP^n_4}(\mathcal{E}^+M_{dP^n_4})$ is open in $M_{dP^n_4}$. Moreover $\phi_{dP^n_4}(\mathcal{E}^+M_{dP^n_4})$ is also closed: let $[X_i]\rightarrow [X_\infty]$ in $M_{dP^n_4}$, with $X_i$ KE. Then, by continuity of the map, $\phi_{32}([X_{i_j}, \omega_{i_j}]) \rightarrow  \phi_{dP^n_4}([Y_\infty, \omega_{\infty}])= [Y_\infty] \in  \overline{M}^{GIT}_{dP^n_4}$, with $Y_\infty$ a (a-priori possibly singular) KE GH limit.   Since $\overline{M}^{GIT}_{dP^n_4}$ is Hausdorff and by definition $\phi_{dP^n_4}([X_{i_j}, \omega_{i_j}])=[X_{i_j}]$, the limit is unique, hence $X_\infty \cong Y_\infty$ admits a KE metric.  $M_{dP^n_4}$ is also connected, hence $\phi_{dP^n_4}(\mathcal{E}^+M_{dP^n_4})=M_{dP^n_4}$, recovering $\cite{AGP}$. Finally $\phi_{dP^n_4}(  \overline{\mathcal{E}^+M}^{GH}_{dP^n_4}) $ is a compact set containing the dense subset $M_{dP^n_4}$. Hence $\phi_{dP^n_4}$ has to be onto. Being $\phi_{dP^n_4}$ a continuous map between a compact and an Hausdorff space, $\phi_{dP^n_4}$ is also open. This conclude the proof of our main Theorem \ref{MT1}.

 \

We now describe some properties of such KE moduli space and of their boundary points.  By the classification of GIT polystable intersection of two quadrics \cite{AL} (Theorem $4.2$ and its proof), we get the following result on the explicit existence of (weak) KE metric on  intersections of two quadrics in $\C\P^{n+2}$.

\begin{thm}\label{KEQ}
	A possibly singular complete intersection of two quadrics in $\C\P^{n+2}$ admits a KE metric if and only if the two quadrics can be simultaneusly diagonalized, the discriminant of their pencil admits no roots of multiplicity $>\frac{n+3}{2}$ and if there is a root of multiplicity exactly equal to $\frac{n+3}{2}$ then $Q_1 \cap Q_2$ is actually isomorphic to $\{\sum_{i=0}^{\frac{n+1}{2}} x_i^2=\sum_{i=\frac{n+3}{2}}^{n+2} x_i^2=0\}$.
\end{thm}

The above just says that, by changing basis, KE intersections are cut  by $Q_1= Id$ and $Q_2$ diagonal with eigenvalues whose multiplicities are no greater than $\frac{n+3}{2}$. Note that the equality $\frac{n+3}{2}$ case for the multiplicity can occur for odd dimensional intersections of two quadrics only.

In dimension three, in particular, we can say that GH limits have at most isolated Hermitian singularities or, if not, they coincide with the orbifold $ \C\P^3/\Z_2$, which is singular along two disjoint smooth rational curves (this a rigidity case for Liu's volume estimate). It is also easy to see that the number of ODP singularities is always even, their maximum number is six, and  this is achieved by the unique toric KE intersection $xy=zt=uv$. Moreover, thanks to the relation with the well-understood invariants of binary sextics, we have  $\overline{\mathcal{E}^+M}^{GH}_{dP^3_4}\cong \overline{\mathcal{K}M}_{dP^3_4}\cong \C\P(1,2,3,5)$ as topological spaces, with boundary divisor given by $t=0$, where $t$ is the weight $5$ coordinate in the weighted projective space.

More generally, from Theorem \ref{KEQ}, we can get the following understanding of the singular set:

\begin{cor} \label{sing}  The singular set of GH limits of KE intersections of two quadrics in $\C \P^{n+2}$ is at most of complex dimension $\lfloor\frac{n-1}{2}\rfloor$. Moreover, the algebraic stratification of the singular set consists of disjoint smooth strata  which are bundles of ODP singularities, i.e., locally analytically of type $\C^k \times A_1^{n-k}$ with $k\in \{0, \dots, \lfloor\frac{n-1}{2}\rfloor\}$. In particular,  there are no quotient singularities as soon as $n\geq4$.
\end{cor}

Clearly the description is so explicit that one can even compute the number of connected components of the singular set, as well as describe their topology. 

Finally a couple of remarks are needed.

\begin{rmk}
	It may be interesting to observe that the above KE spaces have isolated ODP singularities which degenerate to non-isolated singularities (here still of type $\C^k \times A_1^{n-k}$) in their GH limits: locally analytically this corresponds to deformations of singularities like:
	$$V_t: \; x_0^2+x_1^2+x_2^2+tx_3^2=0.$$ 
	Such behavior shares some similarities with the jumping of metric tangent cones described in Section $2$. 
	
	The above phenomenon already happens for certain deformations of the three dimensional KE orbifold $ \C\P^3/\Z_2$. Note also that, in such case,  stability excludes the possibility of obtaining  a KE smoothing of this KE orbifold with  \emph{only one} of the two rational curves of singularities smoothed out. Clearly, similar behaviors occur in higher dimensions too.
\end{rmk}
\begin{rmk}
	It is natural to conjecture that the above singular KE metrics admit asymptotic expansions in suitable local holomorphic gauge to products of flat times Stenzel's cone metrics (i.e., they are conically singular metrics).  We should remark that by \cite{HS} similar results hold for smoothable Calabi-Yau varieties with isolated ODP singularities. It is natural to expect that similar arguments may work also in the KE Fano case, and we leave this for future work.
	\end{rmk}

\subsection{Relations with moduli of  hyperelliptic curves}

 Finally, we briefly emphasize that smooth intersections of two quadrics are also deeply related to moduli spaces  of hyperelliptic curves: to any such \emph{smooth} Fano $n$-fold one can naturally associate a genus two  curve by taking the ramified $2:1$ cover branching over the zeros of the discriminant of the pencil.
 
 When $n=3$,  one can further recover the original Fano as a moduli of stable rank $2$ vector bundles of odd degree on the curve (\cite{Ne, NR}). This gives a well-know instance of a \textquotedblleft Fourier-Mukai type duality" between moduli space of genus two curves and moduli space of our Fano manifolds.
 
  However (see \cite{AL}), this 1-1 relation breaks at the boundaries of their two canonical compactifications: the Deligne-Mumford compactification $\overline{M_2}^{DM}$ (which is the exactly equal to the  K-compactification in the case of curves) is a (smooth) blow-up of our K-compactifiaction $\overline{\mathcal{K}M}_{dP_4^3}$. In particular, $\overline{\mathcal{K}M}_{dP_4^3}$ is more related to a different algebro-geometric compactification of genus two curves, namely to  the so-called $\overline{M}_2[A_2]$ compactification (e.g. the survey \cite{FS}). Such compactification naturally appears in the Hassett-Keel birational study of the moduli space of curves.  
  
  More precisely, the coarse varieties  $\overline{\mathcal{K}M}_{dP_4^3}$ and $\overline{M}_2[A_2]$ are \emph{isomorphic} to each other and are both categorical quotients of some good moduli spaces in Alper's sense: of the \emph{KE-K moduli stack}  of K-semistable degree four del Pezzo threefolds  (see \cite{LWX1, O15, OSS} for details on the definitions), and of the stack of pseudostable curves (nodal curves where elliptic tails are replaced by cusps singularities) for $\overline{M}_2[A_2]$. In particular, there is a set theoretical 1-1  map between closed points in the stacks (the relation become more complicated on the full semistable strata). For example, the orbifold $\C\P^3/\Z_2$ corresponds to the \emph{bicuspidal rational curve}, i.e., the rational curve with only two cuspidal singularities of type $x^2=y^3$.  It is not clear how the vector bundle relation extends to the boundary.  We think that such correspondences deserve further investigation.
  
 For higher dimensions, here we just recall that, by the well-known theorem of Desale and Ramanan \cite{DR} there is still an identification between moduli spaces of stable rank two vector bundles with fixed determinant on smooth hyperelliptic curves of genus $g$  and the variety of $g-2$ dimensional linear subspaces of the corresponding simultaneously  diagonalized smooth intersection of two quadrics.

 \section{Further KE moduli examples: conjectures} \label{OtherKE}
 
 In this last section we discuss further examples of KE moduli under the hypothesis of a natural gap conjecture about volume of CY cones/Sasaki-Einstein manifolds. Moreover, we formulate an analogous conjecture on the algebro-geometric counterpart.

 \subsection{ODP volume gap conjecture}
 
 As it is apparent from the discussion in previous sections, the volume of the link of tangent cones plays an important role. In dimension two, since the singularities occurring in GH limits are of orbifold type, it is obvious that such volume densities are always of type $1/|G_p|$ where $|G_p|$ is the order of the orbifold group at $p$. In particular, one sees immediately that there is a \emph{gap} between smooth or singular points where the density is always $\Theta(Z,p)\leq\frac{1}{2}$ and the value $\frac{1}{2}$ precisely occurs for the $A_1$ singularity, i.e., the ordinary double point singularity, which can be considered in some sense as ``the simplest singularity" in dimension two. 
 
 In higher dimensions non-quotient singularities may appear in GH limits, and the metric tangent cone may not be isomorphic to the singularity itself. Nonetheless one can still consider the volume density at a singular point $p$, and we still have the following well-known abstract gap result:
 
 \begin{lem} \label{lem5-1}
 There is a constant $\delta<1$ depending only on $n$, such that for any GH limit $Z$ of KE Fano manifolds in dimension $n$, we have $\Theta(Z, p)\leq\delta$ for all singular points $p\in Z$. 
 \end{lem}
 \begin{proof}
 If the tangent cone $C(Y)$ has smooth cross section $Y$, then $Y$ is Sasaki-Einstein and this follows from Anderson's rigidity result \cite{Anderson}. In general one needs to use in addition Colding's theorem on volume convergence under GH limits \cite{Colding}. Indeed, if $p$ has volume density sufficiently close to $1$, then we can find a ball centered at $p$ with radius $r$ sufficiently small such that $Vol(B(p, r))/\omega_{2n}r^{2n}$ is also close to $1$. By volume convergence, we can find corresponding balls $B(p_i, r)$ on the smooth manifolds before we take limits, with $Vol(B(p_i, r))/\omega_{2n}r^{2n}$ close to $1$, hence by Anderson's gap result one obtains that $B(p_i, r)$ converges smoothly to $B(p, r)$ (see Theorem 0.8 in \cite{Colding}). 
 \end{proof}
 
 Now for $n\geq 2$ we define  $\mathcal{A}(n)$ to be the supremum of $\Theta(Z, p)$ for all GH limits $Z$ of KE Fano manifolds in dimension $n$, and for all singular points $p\in Z$.  By the above lemma $\mathcal{A}(n)<1$, and $1-\mathcal{A}(n)$ can be viewed as the optimal volume gap in dimension $n$. Again by Colding's volume convergence theorem $\mathcal{A}(n)$ can be realized as the density of some pointed Gromov-Hausdorff limit metric cone of a sequence of smooth KE Fano manifolds. 
 
 For our purpose here we only need a slightly different notion, without referring to singular GH limits. Namely we define $\mathcal{A}'(n)$ to be the supremum of $\Theta(C(Y))$, where $C(Y)$ is a cone of the form $\C^{n-k}\times C(Y')$ for $C(Y')$  a $k$ dimensional CY cone which is singular precisely at the vertex. Clearly $\mathcal{A}'(n)\leq \mathcal{A}(n)\leq 1$, and $\mathcal{A}'(m)\leq \mathcal{A}'(n)$ for $m\leq n$. Moreover, by \cite{MSY}, $\A'(n)$ is always an algebraic number. 
 
 For the simplest possible singularity, namely, the $A_1$ ODP  singularity $\sum_{i=1}^{n+1}x_i^2=0$ (which is not of quotient type as soon as $n>2$), we know it admits a  natural Stenzel's  CY cone metric. The corresponding Sasaki-Einstein manifold is a circle bundle over the quadric hypersurface in $\P^{n}$ equipped with a homogeneous K\"ahler-Einstein metric. It is easy to calculate that the volume density is $2(1-\frac{1}{n})^n$. Hence we have a lower bound
\begin{equation} \label{eqn5-1}
\mathcal{A}'(n)\geq 2(1-\frac{1}{n})^n. 
\end{equation}

\begin{thm} \label{thm5-2}
Let $Z$ be a GH limit of KE Fano manifolds $X_i$ with Fano index $r$ and 
$$c_1^n(-K_{X_i})>\frac{1}{2}\mathcal{A}'(n)(n+1)^n,$$ then $Z$ has Gorenstein canonical singularities and $-K_Z=rL_Z$ for some Cartier divisor $L_Z$. 
\end{thm}
 
\begin{proof}
The proof follows similar lines as in Section 3 and 4.  Below we only sketch the new technical points.

\begin{lem}
For any iterated tangent cone of the form $\C^{k}\times C(Y)$ with $k\geq 0$, $\pi_1(Y)$ is trivial unless the cone is a quotient of the form $\C^n/G$ for a finite subgroup $G\subset U(n)$. 
\end{lem}

\begin{proof}
Suppose $\pi_1(Y)$ is not trivial, then  arguing as in the proof of Lemma \ref{lem3-3}, we obtain the universal cover $\widetilde Y$ with $Vol(\widetilde Y)>\mathcal{A}'(n)Vol(S^{2n-2k-1})$. We claim that $\widetilde Y$ must be smooth. Otherwise, by Remark \ref{rmk3-4} there is a point $q\in Y\subset C(Y)$ with $\Theta(C(Y), q)>\mathcal{A}'(n)$. By passing to an iterated tangent cone we obtain a cone of the of form $\C^l\times C(Y')$ with $Y'$ smooth but volume density bigger than $\mathcal{A}'(n)$. Contradiction. Now given the claim, by our  gap assumption we  conclude that the cone over $\widetilde Y$ must be the flat $\C^n$, hence $C(Y)$ is a quotient of $\C^n$. 
\end{proof}

By (\ref{eqn5-1}), $\mathcal{A}'(n)>2/3$ for $n\geq 6$ and $\mathcal{A}'(n)>1/2$ for all $n$, it follows that the above subgroup $G$ can only be either $\Z_2$ or $\Z_3$.  Now one can follow the induction proof of Lemma \ref{cart}. If at $p$ all the iterated tangent cones are not equal to quotient of $\C^n$,  then the proof is exactly the same. If one iterated tangent cone is a quotient $\C^n/G$, then for the induction argument to go through, we need to follow the proof of Proposition \ref{nonisolated}. Notice such quotient is necessarily of the form $\C^{n-l}\times \C^{l}/G$, where $G$ acts non-trivially in the sphere $S^{2l-1}\subset \C^{l}$. We claim $l=2$. This follows by taking generic sections as in the proof of Proposition \ref{nonisolated}, and using Schlessinger's rigidity of quotient singularities in dimension greater than two \cite{Sc}. Now for our induction to work by the proof of Proposition \ref{nonisolated}, it suffices to show that on a smoothing of $\C^2/G$ with $G=\Z_2$ or $\Z_3$, there is a Stein neighborhood of the vanishing cycle that has trivial fundamental group. This is clearly the case when $G\subset SU(2)$. So we are left with the case $G=\Z_3$ acts diagonally on $\C^2$ with weight $(\zeta_3, \zeta_3)$. We claim this case never occurs. Notice it is is the non-$\Q$-Gorenstein smoothable cone over a rational normal curve of degree $3$. Its versal deformation space (coinciding with the Artin component \cite{St}) is given by a two dimensional space which parameterizes smoothings admitting  Stein neighborhoods homeomorphic to the  total space of the $\mathcal{O}(-3)$ bundle over $\P^1$, hence again simply connected. So by the proof of Proposition \ref{nonisolated} it follows that $\C^2/G$ has to be Gorenstein, so we have proved the claim.
\end{proof}

When we use the rough bound $\mathcal{A}'(n)\leq 1$ the above result reduces to Corollary \ref{cor3-6}. In general for practical applications of Theorem \ref{thm5-2} it is important to understand $\A'(n)$ better.  It is not unreasonable to expect the following:

\begin{conj}\label{conj5-4}
For all $n$, 
$$\A'(n)\leq 2(1-\frac{1}{n})^n. $$
\end{conj}
\noindent One expects the equality only holds for the $n$ dimensional Stenzel cone. Now it is easy to see that Conjecture \ref{DGODP} implies Conjecture \ref{conj5-4}, and the first part of Theorem \ref{thm1-3} follows exactly from Theorem \ref{thm5-2}.

\

 Volume of links of CY cones have been studied by Martelli-Sparks-Yau in the context of AdS-CFT correspondence (see, e.g., \cite{MSY}), where it is shown that, CY cones are minimizers of the volume functional on the space of Reeb vector fields.   From a physics point of view, here we just remark that, in dimension three, such volume densities are supposed to be inverses of quantities related to central charges of certain dual superconformal field theory (\emph{a-maximization}). 
 
 As we have recalled in Section \ref{prelim}, thanks to the description of the volume minimization in terms of the language of valuations \cite{Li16} and their relation to densities (still conjecturally in the irregular case), it is natural to formulate a purely algebro-geometric conjecture analogous to the density gap above described.
 
  \begin{conj}[ODP volume gap conjecture, algebraic version] \label{AlgODP}
  	Let $(V,p)$ be a non-smooth germ of a $n$-dimensional klt singularity. Then the infimum of the normalized volume of valuations centered at $p$ satisfies  $$\inf_{\nu} \widehat{vol}(\nu) \leq 2 \left( n-1\right)^n,$$
  	with equality realized only by the ordinary double point singularity $\sum_{i=1}^{n+1}x_i^2=0$.
  \end{conj}

  \begin{rmk}\begin{enumerate}
\item From discussions with Kento Fujita and Yuchen Liu, we  are informed that even the weaker inequality $\inf_{\nu} \widehat{vol}(\nu) \leq n^{n}$ is not yet proved in complete generality. Notice from a metric point of view, if $(V, p)$ is a germ of a pointed rescaled Gromov-Hausdorff limit of KE Fano manifolds, then by Proposition \ref{val} this weaker inequality holds by Bishop-Gromov volume comparison theorem. On the other hand, Conjecture \ref{conj5-4} follows from Conjecture \ref{AlgODP} if the equality case of Proposition \ref{val} holds (which is known to be true when the metric tangent cone is quasi-regular).
  		\item When $n=3$, it is possible that such densities are always $\leq \frac{1}{2}$ except the cases of $A_1$ singularity, when it is equal to $16/27$ or $A_2$ singularity, when it is equal to $125/243$ (note that they are equal to $1/2$ for any other $A_k$ singularity for $k\geq3$, corresponding to the expected jumping picture for the metric tangent cone to the flat cone  $\C\times \C^2/ \Z_2$). In fact, we do not know any other CY cones (even among known quasi-regular or irregular ones) which have higher values of densities.
		
  	\end{enumerate}
  	
  \end{rmk}

 \subsection{Del Pezzo of degree three (cubic hypersurfaces)}
 
Now we prove part (2) of Theorem \ref{thm1-3}.

Let $Z$ be a GH limit of smooth K\"ahler-Einstein cubic hypersurfaces in dimension $n\geq 3$. Since the volume is given by $3(n-1)^n\geq (1-\frac{1}{n})^n(n+1)^n$, by Theorem  \ref{thm5-2} we know $Z$ has at worst Gorenstein canonical singularities and $K_Z=L_Z^{n-1}$ for some ample line bundle $L_Z$. Hence by Fujita's classification \cite{TFuj} $Z$ must be a cubic hypersurface too.  Notice also by \cite{T00} (or \cite{N} for $n\leq4$) a Fermat cubic hypersurface $x_0^3+\cdots+x_n^3=0$ admits a K\"ahler-Einstein metric so is K-stable. By the same discussion as in \cite{OSS} this implies that the CM line bundle over the GIT moduli is a positive multiple of the standard polarization $\O(1)$, hence is ample. By \cite{B} we know $Z$ is K-polystable, so $Z$ is a GIT polystable cubic hypersurface. Now we argue as in \cite{OSS} and conclude that the GIT compactification agrees with KE compactification. This finishes the proof of Theorem \ref{thm1-3}, (2). 

Now we recall that in dimension 3 and 4 the GIT moduli space has been extensively studied. GIT of cubic threefolds is done by Allcock \cite{All}. We recall the result
 
 \begin{thm}[\cite{All}]
 \begin{itemize}
 \item A cubic threefold is GIT stable if and only if it has only isolated singularities of type $A_k$, $k\leq 4$. 
 \item A cubic threefold is GIT polystable with non-discrete stabilizer if and only if it is isomorphic to $F_\Delta=x_0x_1x_2+x_3^3+x_4^3$ which has exactly three $D_4$ singularities; or it is isomorphic to $F_{A, B}=Ax_2^3+x_0x_3^2+x_1^2x_4-x_0x_2x_4+Bx_1x_2x_3$ where at least one of $A, B$ is non-zero, which has at worst isolated $A_1$ and $A_5$ singularities when $4A\neq B^2$, and which has a transverse $A_1$ singularity along a rational normal curve when $4A=B^2$(the chordal cubic). 
 \end{itemize}
 \end{thm}
 
 In particular it is easy to see that a GIT polystable cubic in dimension 3 must have canonical singularities. 
 
GIT of cubic fourfolds has been studied in detail by Laza in \cite{Laza}. In particular, it is known that a cubic fourfold with ADE singularities is GIT stable, and singularities of a polystable  cubic fourfold corresponding to a generic point on the GIT quotient compactification are classified (see the longer list of such possible singularities in Table $3$ of \cite{Laza}).
 
 As far as we are aware, there is no explicit study of higher dimensional GIT quotients of cubics hypersurfaces. However, our previous discussion leads to some purely GIT questions that may be worth studying: as a direct consequence of Theorem \ref{thm1-3}, (2), we notice that if Conjecture \ref{conj5-4} holds, then the GIT moduli agrees with KE moduli. However, by \cite{DS14} every GH limit in the KE moduli must have log terminal singularities, and if is again a cubic hypersurface then it must have Gorenstein canonical singularities. So a natural question is:

 \begin{prob}
Is it true that a GIT polystable cubic in $\P^n$ has at worst Gorenstein canonical singularities? In particular, does it always have normal singularities?
 \end{prob}
 
This is intimately related to Conjecture \ref{conj5-4}. If the answer to this question is negative, then it would mean Conjecture \ref{conj5-4} is false.

We observe that even without assuming Conjecture \ref{conj5-4} we have: 

\begin{prop}
A GH limit $Z$ of smooth K\"ahler-Einstein cubic threefolds has only Gorenstein canonical singularities, and $K_Z^{-1}=L_Z^{2}$ for some $\Q$-line bundle $L_Z$. 
\end{prop}

This follows from similar arguments as in the proof of Lemma \ref{cart} and Proposition \ref{nonisolated}, using the fact that $K_{X_i}^{-1}=L_{X_i}^2$. If one can prove $L_Z$ is a genuine line bundle, then it would follow from the classification of Fujita \cite{TFuj} that $Z$ must be a cubic hypersurface in $\P^4$ and hence the KE moduli agree with GIT moduli. It would also be interesting to see if Fujita's classification can be extended to our setting.

  \subsection{Other Fano manifolds}
 
 We expect that, in order to proceed further, an extension of the classification of mildly singular smoothable Fano $3$-folds with small Gorenstein index and with bounds on invariant of the singularities given by differential geometric considerations as  above, would be useful for studying both moduli space of Fano 3-folds and the existence of KE metrics on them. We should recall that there is a classification of Gorenstein Fano $3$-folds with assumptions on decompositions of the anticanonical linear system by Mukai \cite{Mukai}. 
 
   
   Our discussion gives the following picture regarding a-priori bounds on the Gorenstein index of GH limits of KE Fano $3$-folds:
   
   \begin{cor} Assuming Conjecture \ref{conj5-4}, GH limits of smooth  KE Fano  $3$-folds whose volume is $\geq 20$  are always Gorenstein with canonical singularities. Without assuming Conjecture \ref{conj5-4}, the Gorenstein index of GH limits is at most two as long as the volume is $\geq 22$ and in this case the canonical divisor is Gorenstein away from at most finitely many points. 
   	
   	Among the $105$ deformation types of smooth Fano $3$-folds, $75$ have volume bigger than or equal to $20$, and four deformation classes have volume equal to $22$.
   	\end{cor}

Here we recall that the volume of a Fano $3$-fold is always even, and  as long as the volume is bigger than $32$ we have mostly rigid examples.  
 
 Interesting examples  include the case of Iskovskikh-Mukai-Umemura Fano 3-folds of degree $22$ and Picard rank one, since arguments along these lines can solve the problem of understanding which smooth ones are KE, even ``far away" from the special KE  Mukai-Umemura $3$-fold. Here it would be interesting to show that all GH limits are still given by  intersections of three  sections of the tautological vector bundle over the Grassmanian as the smooth ones do.   Then  (see discussion in \cite{Sp}),  it seems very plausible that one can consider a GIT picture similar to the previous case of degree four del Pezzos to conclude.
 
 Another example (with much smaller anticanonical volume) we like to briefly discuss is the case of quartic hypersurfaces in $\C\P^4$. By very recent work of Fujita \cite{Fujita1}, it is known that all smooth quartics admit KE metrics. However, by classification of smooth  Fano $3$-folds, it is known that quartic hypersurfaces do not form a complete family, since they can be deformed to the ``hyperelliptic'' Fanos  given by double covers of a  smooth quadric in $\P^4$ ramified over a smooth divisor of degree eight obtained as intersection with a quartic. In such situation, it is  also known that all such  Fanos admit KE metrics \cite{AGP}. Thus the GH moduli compactification is \emph{not}  equal to the GIT quotient of quartics, since  we need (at least) to blow-up such GIT quotient at the non-reduced double conic, similarly to the case of degree $2$ del Pezzo considered in \cite{OSS}. We think that a concrete algebro-geometric analysis of such situation, by trying to explicitly construct a compactification made only by $\Q$-Fano varieties, is indeed very interesting.
 
Similarly, moduli of del Pezzos of low degree and of other Fano $3$-folds could be  possibly analyzed by considering GIT gluings coming from certain hypersurfaces in weighted projective spaces. In general, it is clearly very interesting to investigate precisely which Fano 3-folds are known to have a non-empty KE moduli spaces.

 \end{document}